\documentclass{article}
\usepackage{amsmath,amsthm,amssymb,mathrsfs,amstext, titlesec,enumitem}
\makeatletter

\titleformat{\section}
{\normalfont\Large\bfseries}{\thesection}{1em}{}
\titleformat*{\subsection}{\Large\bfseries}
\titleformat*{\subsubsection}{\large\bfseries}
\titleformat*{\paragraph}{\large\bfseries}
\titleformat*{\subparagraph}{\large\bfseries}

\renewcommand{\@seccntformat}[1]{\csname the#1\endcsname. }
\makeatother
\theoremstyle{plain}
\newtheorem{thm}{Theorem}[section]
\newtheorem{lem}[thm]{Lemma}
\theoremstyle{definition}
\newtheorem{defn}[thm]{Definition}
\newtheorem{rem}[thm]{Remark}
\providecommand{\keywords}[1]{{\bf{Keywords:}} #1}
\providecommand{\subjectclass}[2]{\textbf{Mathematics subject classification 2020:} #1}

\title{Dynamical characterization of central sets along filter
\footnote{\keywords{Algebra in the Stone-\v{C}ech compactification, Topological dynamics,
Central set}}}

\date{}
\author{pintu debnath
\footnote{Department of Mathematics, Basirhat College, Basirhat -743412, North
24th parganas, West Bengal, India.
           {\bf pintumath1989@gmail.com}}
           \and
Sayan Goswami
\footnote{Department of Mathematics, 
          University of Kalyani, 
          Kalyani-741235,
          Nadia, West Bengal, India
          {\bf sayan92m@gmail.com}}
}

\begin{document}
\maketitle

\begin{abstract}
\noindent Using the notions of Topological dynamics, H. Furstenberg
defined central sets and proved the Central Sets Theorem. Later V.
Bergelson and N. Hindman characterized central sets in terms of algebra
of the Stone-\v{C}ech Compactification of discrete semigroup. They
found that central sets are the members of the minimal idempotents
of $\beta S$, the Stone-\v{C}ech Compactification of a semigroup
$\left(S,\cdot\right)$. We know that any closed subsemigroup of $\beta S$
is generated by a filter. We call a set $A$ to be a $\mathcal{F}$-
central set if it is a member of a minimal idempotent of a closed
subsemigroup of $\beta S$, generated by the filter $\mathcal{F}$.
In this article we will characterize the $\mathcal{F}$-central sets
dynamically.
\end{abstract}
\subjectclass{37B05, 05D10}

\section{Introduction}

H. Frustenberg introduced the notion of central sets \cite[Defination 8.3]{key-2}
and proved several combinatorial properties of such sets using topological
dynamics. Later, V. Bergelson and N. Hindman in \cite{key-1}, established
an algebraic characterization of central sets. For arbitrary semigroup,
the interplay of central sets between algebra of the Stone-\v{C}ech
Compactification and Topological dynamics was explored in \cite{key-4-4}.
\begin{defn}
A dynamical system is a pair $\left(X,\langle T_{s}\rangle_{s\in S}\right)$
such that
\begin{enumerate}
\item $X$ is a compact topological space;
\item $S$ is a semigroup;
\item for each s $s\in S$, $T_{s}$ is a continuous function from $X$
to $X$; and
\item For all $s,t\in S$, $T_{s}\circ T_{t}=T_{st}$.
\end{enumerate}
\end{defn}

\noindent For any discrete semigroup $\left(S,\cdot\right)$, central set was
defined as the member of the minimal idempotents of its Stone-\v{C}ech
Compactification, say $\beta S$. To state dynamical characterization
of central sets, we need the following definitions.
\begin{defn}
Let $S$ be a discrete semigroup and $\left(X,\langle T_{s}\rangle_{s\in S}\right)$
be a dynamical system.
\begin{enumerate}
\item Let $A$ be a subset of $S.$ Then the set $A$ is called syndetic
if and only if there exists a finite subset $F$ of $S$ such that
$S=\cup_{t\in F}t^{-1}A$.
\item A point $x\in X$ is uniformly recurrent point if and only if for
each neighbourhood $U$ of $x$, $\left\{ s\in S:T_{s}x\in U\right\} $
is syndetic.
\item $x,y\in X$ are called proximal if and only if for every neighbourhood
$U$ of the diagonal in $X\times X$, there exists $s\in F$ such
that $(T_{s}\left(x\right),T_{s}(y)\in U$.
\end{enumerate}
\end{defn}

\noindent From \cite[Theorem 2.4]{key-4-4}, we know the following theorem.
\begin{thm}
Let $S$ be a semigroup and let $B\subseteq S$. Then $B$ is central
if and only if there exists a dynamical system $\left(X,\langle T_{s}\rangle_{s\in S}\right)$,
two points $x,y\in X$, and a neighbourhood $U$ of $y$ such that
$x$ and $y$ are proximal, $y$ is uniformly recurrent and $B=\left\{ s\in S:T_{s}\left(x\right)\in U\right\} $.
\end{thm}

We will extend this result for certain filters $\mathcal{F}$, over
any discrete semigroup $S$, which generates a closed subsemigroup
in the space of ultrafilters, say $\beta S$ (see preliminaries section).
A sets is $\mathcal{F}$-central if it is a member of any minimal
idempotent of the semigroup generated by $\mathcal{F}$ (see definition \ref{Fcen}). Recently
in \cite{key-3}, the Central Sets Theorem along some filters has
been established. Before stating our main theorem, let us define the
analogous notion of uniformly recurrent and proximality along a filter
$\mathcal{F}.$
\begin{defn}
\label{fdym} Let $S$ be a semigroup, and $\left(X,\langle T_{s}\rangle_{s\in S}\right)$
be a dynamical system. Let $T$ be a closed subsemigroup of $\beta S$
such that for filter $\mathcal{F}$, $\overline{\mathcal{F}}=T$.
\begin{enumerate}
\item $A\subseteq S$ is $\mathcal{F}$-syndetic if for every $F\in\mathcal{F}$,
there is a finite set $G\subseteq F$ such that $G^{-1}A\in\mathcal{F}$.
\item A point $x\in X$ is $\mathcal{F}$-uniformly recurrent point if and
only if for each neighbourhood $U$ of $x$, $\left\{ s\in S:T_{s}x\in U\right\} $
is $\mathcal{F}$-syndetic.
\item $x,y\in X$ are $\mathcal{F}$-proximal if and only if every neighbourhood
$U$ of the diagonal in $X\times X$ and for each $F\in\mathcal{F}$
there exists $s\in F$ such that $(T_{s}\left(x\right),T_{s}(y)\in U$.
\end{enumerate}
The following is our one:
\end{defn}

\begin{thm}
\label{main} Let $S$ be a semigroup and let $B\subseteq S$. Then
$B$ is $\mathcal{F}$-central if and only if there exists a dynamical
system $\left(X,\langle T_{s}\rangle_{s\in S}\right)$ and there exist
$x,y\in X$ and a neighbourhood $U$ of $y$ such that $x$ and $y$
are $\mathcal{F}$- proximal, $y$ is $\mathcal{F}$-uniformly recurrent
and $B=\left\{ s\in S:T_{s}\left(x\right)\in U\right\} $.
\end{thm}

\section{Preliminaries}

Let us first give a brief review of algebraic structure of the Stone-\v{C}ech
compactification of any discrete semigroup $S$.

The set $\{\overline{A}:A\subset S\}$ is a basis for the closed sets
of $\beta S$. The operation `$\cdot$' on $S$ can be extended to
the Stone-\v{C}ech compactification $\beta S$ of $S$ so that$(\beta S,\cdot)$
is a compact right topological semigroup (meaning that for any \ 
is continuous) with $S$ contained in its topological center (meaning
that for any $x\in S$, the function $\lambda_{x}:\beta S\rightarrow\beta S$
defined by $\lambda_{x}(q)=x\cdot q$ is continuous). This is a famous
Theorem due to Ellis that if $S$ is a compact right topological semigroup
then the set of idempotents $E\left(S\right)\neq\emptyset$. A non-empty
subset $I$ of a semigroup $T$ is called a $\textit{left ideal}$
of $S$ if $TI\subset I$, a $\textit{right ideal}$ if $IT\subset I$,
and a $\textit{two sided ideal}$ (or simply an $\textit{ideal}$)
if it is both a left and right ideal. A $\textit{minimal left ideal}$
is the left ideal that does not contain any proper left ideal. Similarly,
we can define $\textit{minimal right ideal}$ and $\textit{smallest ideal}$.

 Any compact Hausdorff right topological semigroup $T$ has the smallest
two sided ideal

\[
\begin{array}{ccc}
K(T) & = & \bigcup\{L:L\text{ is a minimal left ideal of }T\}\\
 & = & \,\,\,\,\,\bigcup\{R:R\text{ is a minimal right ideal of }T\}.
\end{array}
\]

\noindent Given a minimal left ideal $L$ and a minimal right ideal $R$, $L\cap R$
is a group, and in particular contains an idempotent. If $p$ and
$q$ are idempotents in $T$ we write $p\leq q$ if and only if $pq=qp=p$.
An idempotent is minimal with respect to this relation if and only
if it is a member of the smallest ideal $K(T)$ of $T$. Given $p,q\in\beta S$
and $A\subseteq S$, $A\in p\cdot q$ if and only if the set $\{x\in S:x^{-1}A\in q\}\in p$,
where $x^{-1}A=\{y\in S:x\cdot y\in A\}$. See \cite{key-4-2} for
an elementary introduction to the algebra of $\beta S$ and for any
unfamiliar details.
\begin{defn}
Let $S$ be a discrete semigroup and let $C$ be a subset of $S$.
Then $C$ is central if and only if there is an idempotent $p$ in
$K\left(\beta S\right)$ such that $C\in p$.
\end{defn}

\noindent For every filter $\mathcal{F}$, on the semigroup $S$, define $\overline{\mathcal{F}}\subseteq\beta S$
by $\overline{\mathcal{F}}=\cap_{F\in\mathcal{F}}\overline{F}$. Note
that $\overline{\mathcal{F}}$ is closed subset of $\beta S$ consisting
of all ultrafilters on $S$ that contain $\mathcal{F}$. Conversely,
every closed subset of $\beta S$ is uniquely represented in such
a form. If $\mathcal{F}$ is idempotent filter, i.e. $\mathcal{F}\supset\mathcal{F}\cdot\mathcal{F}$,
then $\overline{\mathcal{F}}$ becomes a semigroup, but the converse
is not true always. In this article, without mentioned further, we
will consider only those filters $\mathcal{F}$, which generates a
closed subsemigroup of $\beta S$. For details readers can see \cite{key-5-5}.
For some new development in this area we refer \cite{key-1.1}.
\begin{defn} [$\mathcal{F}$-central set]\label{Fcen}
Let $S$ be a discrete semigroup and let $\mathcal{F}$ generates a closed subsemigroup of $\beta S$. Then a set $C$ is said to be $\mathcal{F}$-central if and only if there is an idempotent $p$ in
$K\left(\overline{\mathcal{F}}\right)$ such that $C\in p$.
\end{defn}
\begin{rem}
For $\beta S$ , we know that $\mathcal{F}=\left\{ S\right\} $. Then
$\mathcal{F}$-central sets are nothing but the usual central sets.
Let $S$ is a dense subsemigroup of $\left(\left(0,\infty\right),+\right)$,
and $0^{+}\left(S\right)=\left\{ p\in\beta S:\,\text{for}\,\text{any}\,\epsilon>0,S\cap\left(0,\epsilon\right)\in p\right\} $.
Then $\mathcal{F}=\left\{ \left(0,\epsilon\right)\cap S:\epsilon>0\right\} $
and $\overline{\mathcal{F}}=0^{+}\left(S\right)$. A set $C\subseteq S$
is central set near zero if and only if there is some idempotent $p\in K\left(0^{+}\left(S\right)\right)$
with $C\in p$. So in this case $\mathcal{F}$-centrals are central
sets near zero. See \cite{key-4-1} for combinatorial results of central
sets near zero and \cite[Theorem 2.12]{key-4-3} for dynamical characterizations
of central sets near zero.
\end{rem}

\noindent If $\left(X,\langle T_{s}\rangle_{s\in S}\right)$ be a dynamical
system, then $\overline{\left\{ T_{s}:s\in S\right\} }$ in $X^{X}$
is a semigroup, which is referred as enveloping semigroup of the dynamical
system. Now we recall the following theorem.
\begin{thm} [{\rm \cite[Theorem 19.11]{key-4-2}}]
Let $\left(X,\langle T_{s}\rangle_{s\in S}\right)$
be a dynamical system and define $\theta:S\rightarrow X^{X}$ by $\theta\left(s\right)=T_{s}$.
Then is a continuous homomorphism from $\beta S$ onto the enveloping
semigroup of $\left(X,\langle T_{s}\rangle_{s\in S}\right)$.( $\widetilde{\theta}$
be a continuous extension of $\theta$.)
\end{thm}

\noindent Let us recall the definition \cite[Definition 19.12]{key-4-2}, which
will be useful.
\begin{defn}
Let $\left(X,\langle T_{s}\rangle_{s\in S}\right)$ be a dynamical
system and define $\theta:S\rightarrow X^{X}$ by $\theta\left(s\right)=T_{s}$.
For each $p\in\beta S$, let $T_{p}=\widetilde{\theta}\left(p\right)$.
\end{defn}

\noindent As an immediate consequences of Theorem 8, we have the following remark \cite[Remark 19.13]{key-4-2}.
\begin{rem}
\label{rem} Let $\left(X,\langle T_{s}\rangle_{s\in S}\right)$ be
a dynamical system and let $p,q\in\beta S$. Then $T_{p}\circ T_{q}=T_{pq}$
and for each $x\in X$, $T_{p}\left(x\right)=p-lim_{s\in S}T_{s}(x)$.
\end{rem}

\noindent {\bf Strategy of our proof:} In the next section, we will first establish the relation between $\mathcal{F}$-proximality and the ultrafilters containing $\mathcal{F}$. Then we will establish the relation between algebra and $\mathcal{F}$-uniformly recurrent point. Soon after we will deduce three lemma which will explore the relation between algebra, $\mathcal{F}$-proximal and $\mathcal{F}$ uniformly recurrent point. Then using these results we will obtain our desire result.  
\section{Dynamical characterization of $\mathcal{F}$-central set}

From \cite[Lemma 19.22]{key-4-2}, we get the characterization of proximality
which states that for a dynamical system $\left(X,\langle T_{s}\rangle_{s\in S}\right)$
and $x,y\in S$. Points $x$ and $y$ of $X$ are proximal if and
only if there is some $p\in\beta S$ such that $T_{p}(x)=T_{p}(y)$.
We get an analogous result for $\mathcal{F}$-proximality by the following
lemma, which will be very convenient for us.
\begin{lem}
\label{11} Let $S$ be a semigroup, and $\left(X,\langle T_{s}\rangle_{s\in S}\right)$
be a dynamical system. Then $x,y\in X$ are $\mathcal{F}$-proximal
if and only if there is some $p\in\overline{\mathcal{F}}$ such that
$T_{p}(x)=T_{p}(y)$.
\end{lem}

\begin{proof}
Let $x,y\in X$ are $\mathcal{F}$-proximal. Let $\mathcal{N}$ be
the set of all neighbourhoods of the diagonal in $X\times X$. For
each $U\in\mathcal{N}$, let \textbf{$B_{U}=\left\{ s\in S:\left(T\left(x\right),T_{s}\left(y\right)\right)\in U\right\} $}.
From definition \ref{fdym}, it follows that $\left\{ B_{U}:U\in\mathcal{N}\right\} \cup\mathcal{F}$
has finite intersection property. Now choose $p\in\overline{\mathcal{F}}$
such that $\left\{ B_{U}:U\in\mathcal{N}\right\} \cup\mathcal{F}\in p$.
Let $z=T_{p}\left(x\right)$. To see that $z=T_{p}\left(y\right)$,
let $V$ be an open neighbourhood of $z$ in $X$. Since $X$ is compact
Hausdorff, there exist disjoint open sets $V_{1}$, $V_{2}$ such
that $z\in V_{1}$ and $X\setminus V\subseteq V_{2}$. Let $U=\left(V\times V\right)\cup\left(V_{2}\times X\right)$.
Then $U$ is a neighbourhood of the diagonal in $X\times X$ such that
$\pi_{2}\left(\pi_{1}^{-1}\left(V_{1}\right)\cap U\right)\subseteq V$,
where $\pi_{1}$ and $\pi_{2}$ denote the first and second projections
of $X\times X$ on to $X$ respectively. Let $E=\left\{ s\in S:T_{s}\left(x\right)\in V_{1}\right\} $
and $F=\left\{ s\in S:\left(T_{s}\left(x\right),T_{s}\left(y\right)\right)\in U\right\} $.
Then $E,F\in p$ and $E\cap F\subseteq\left\{ s\in S:T_{s}\left(y\right)\in V\right\} $.
Thus $\left\{ s\in S:T_{s}\left(y\right)\in V\right\} \in p$ for
every open neighbourhood $V$ of $z$.

Conversely suppose, we have $p\in\overline{\mathcal{F}}$ such that
$T_{p}(x)=T_{p}(y)=z$. Let $U$ be a neighbourhood of the diagonal
in $X\times X$. Choose an open neighbourhood $V$ of $z$ in $X$
such that $V\times V\subseteq U$. Let $B=\left\{ s\in S:T_{s}\left(x\right)\in V\right\} $
and $C=\left\{ s\in S:T_{s}\left(y\right)\in V\right\} $. Then $B\cap C\in p.$
For each $F\in\mathcal{F}$, choose $s\in F\cap B\cap C$. Then $\left(T_{s}\left(x\right),T_{s}(y)\right)\in V\times V\subseteq U$.
\end{proof}
In order to establish the equivalence of dynamical and algebraic notions
of central sets, the following theorem is necessary.
\begin{thm}
\label{4eq} Let $S$ be a semigroup and $\left(X,\langle T_{s}\rangle_{s\in S}\right)$
be a dynamical system and $L$ be a minimal left ideal of \textup{$\overline{\mathcal{F}}$}
and $x\in X$. The following statements are equivalent.
\begin{enumerate}[label=(\alph*)]
\item The points $x$ is a $\mathcal{F}$-uniformly recurrent point of $\left(X,\langle T_{s}\rangle_{s\in S}\right)$.
\item There exists $u\in L$ such that $T_{u}(x)=x$.
\item There exist $y\in X$ and an idempotent $u\in L$ such that $T_{u}(y)=x$.
\item There exists an idempotent $u\in L$ such that $T_{u}(x)=x$.
\end{enumerate}
\end{thm}

\begin{proof}
$\left(a\right)\Rightarrow\left(b\right)$ Choose any $v\in L$. Let
$\mathcal{N}$ be the set of neighbourhoods of $x$ in $X$. For each
$U\in\mathcal{N}$, let $B_{U}=\left\{ s\in S:T_{s}\left(x\right)\in U\right\} $.
Since $x$ is a $\mathcal{F}$-uniformly recurrent point, each $B_{U}$
is $\mathcal{F}$-syndetic set. So for every $F\in\mathcal{F}$, there
is some finite set $F_{U,F}\subset F$ such that $F_{U,F}^{-1}B_{U}\in\mathcal{F}\subset v$.
So for each $U\in\mathcal{N}$ and $F\in\mathcal{F}$ pick $t_{U,F}\in F_{U,F}$
such that $t_{U,F}^{-1}B_{U}\in v$. Given $U\in\mathcal{N}$, let
$C_{U}=\left\{ t_{V,F}:V\in\mathcal{N},V\subseteq U\,\text{and}\,F\in\mathcal{F}\right\} $.
Then $\left\{ C_{U}:U\in\mathcal{N}\right\} \cup\mathcal{F}$ has
finite intersection property. So pick $w\in\overline{\mathcal{F}}$
such that $\left\{ C_{U}:U\in\mathcal{N}\right\} \subseteq w$ and
let $u=w\cdot v$. Since $L$ is a left ideal of $\overline{\mathcal{F}}$
, $u\in L$. To see that $T_{u}(x)=x$, we let $U\in\mathcal{N}$
and show that $B_{U}\in u$, for which it suffices that $C_{U}\subseteq\left\{ t\in S:t^{-1}B_{U}\in v\right\} $.
So let $t\in C_{U}$ and pick $V\in\mathcal{N}$ and $F\in\mathcal{F}$
such that $V\subseteq U$ and $t=t_{V,F}$. Then $t^{-1}B_{V}\in v$
and $t^{-1}B_{V}\subseteq t^{-1}B_{U}$.

$(b)\Rightarrow(c)$ Let $K=\left\{ v\in L:T_{v}\left(x\right)=x\right\} $.
It suffices to show that $K$ is a compact subsemigroup of $L$, since
then $K$ has an idempotent. By the assumption, $K\neq\emptyset$.
Further if $v\in L\setminus K$ there is some neighbourhood of $x$
such that $B=\left\{ s\in S:T_{s}\left(x\right)\in U\right\} \notin v$.
Then $\overline{B}\cap L$ is a neighbourhood of $v$ in $\beta S$
which misses $K$. Finally, to see that $K$ is a semigroup, let $v,u\in K$.
Then $T_{v\cdot w}(x)=T_{v}\left(T_{w}\left(x\right)\right)=T_{v}\left(x\right)=x$.

$\left(c\right)\Rightarrow\left(d\right)$ Using remark \ref{rem},
we have $T_{u}(x)=T_{u}\left(T_{u}\left(y\right)\right)=T_{u}\left(y\right)=x$.

$\left(d\right)\Rightarrow\left(a\right)$ Let $U$ be a neighbourhood
of $x$ and let $B=\left\{ s\in S:T_{s}\left(x\right)\in U\right\} $
and suppose that $B$ is not $\mathcal{F}$-syndetic. Then there exists
$F\in\mathcal{F}$ such that
\[
\left\{ S\setminus\cup_{t\in F_{F}}t^{-1}B:F_{F}\,\text{is}\,\text{a}\,\text{finite}\,\text{nonempty}\,\text{subset}\,\text{of}\,F\right\} \cup\mathcal{F}
\]
 has finite intersection property. So pick some $w\in\overline{\mathcal{F}}$
such that 
\[
\left\{ S\setminus\cup_{t\in F_{F}}t^{-1}B:F_{F}\,\text{is}\,\text{a}\,\text{finite}\,\text{nonempty}\,\text{subset}\,\text{of}\,F\right\} \subseteq w.
\]
 Then $\left(\overline{\mathcal{F}}\cdot w\right)\cap\overline{B}=\emptyset$
(As $B\in v\cdot w$ implies $t^{-1}B\in w$ for some $t\in F$).
Now $\left(\overline{\mathcal{F}}\cdot w\right)$ is a left ideal
of $\overline{\mathcal{F}}$, so $\overline{\mathcal{F}}\cdot w\cdot u$
is a left ideal of $\overline{\mathcal{F}}$ which is contained in
$L$, and hence $\overline{\mathcal{F}}\cdot w\cdot u=L$. Thus we
may pick some $v\in\overline{\mathcal{F}}\cdot w$ such that $v\cdot u=u$.
Again $T_{v}(x)=T_{v}\left(T_{u}\left(x\right)\right)=T_{v\cdot u}\left(x\right)=T_{u}\left(x\right)=x$,
so in particular $B\in v$. But, $v\in\overline{\mathcal{F}}\cdot w$
and $\left(\overline{\mathcal{F}}\cdot w\right)\cap\overline{B}=\emptyset$,
a contradiction.
\end{proof}
\begin{lem}
Let $S$ be a semigroup and $\left(X,\langle T_{s}\rangle_{s\in S}\right)$
be a dynamical system and let $x\in X$. Then for each $F\in\mathcal{F}$,
there is a $\mathcal{F}$-uniformly recurrent point $y\in\overline{\left\{ T_{s}\left(x\right):s\in F\right\} }$
such that $x$ and $y$ are $\mathcal{F}$- proximal.
\end{lem}

\begin{proof}
Let $L$ be any minimal left ideal of $\overline{\mathcal{F}}$ and
pick an idempotent $u\in L$. Let $y=T_{u}\left(x\right)$. For each
$F\in\mathcal{F}$, clearly $y\in\overline{\left\{ T_{s}\left(x\right):s\in F\right\} }$,
as $\mathcal{F}\subseteq u$.By lemma \ref{11}, $y$ is a $\mathcal{F}$-uniformly
recurrent point of $\left(X,\langle T_{s}\rangle_{s\in S}\right)$.
By remark \ref{rem}, we have $T_{u}(y)=T_{u}\left(T_{u}\left(x\right)\right)=T_{u\cdot u}\left(x\right)=T_{u}\left(x\right)$.
So by Lemma \ref{11}, $x$ and $y$ are $\mathcal{F}$- proximal.
\end{proof}
\begin{lem}
\label{15} Let $S$ be a semigroup and $\left(X,\langle T_{s}\rangle_{s\in S}\right)$
be a dynamical system and let $x,y\in X$. If $x$ and $y$ are $\mathcal{F}$-
proximal, then there is a minimal left ideal $L$ of $\overline{\mathcal{F}}$
such that $T_{u}\left(x\right)=T_{u}\left(y\right)$ for all $u\in L$.
\end{lem}

\begin{proof}
Pick $v\in\overline{\mathcal{F}}$ such that $T_{v}\left(x\right)=T_{v}\left(y\right)$
and pick a minimal left ideal $L$ of $\overline{\mathcal{F}}$ such
that $L\subseteq\overline{\mathcal{F}}\cdot v$. To see that $L$
is as required, let $u\in L$ and choose $w\in\overline{\mathcal{F}}$
such that $u=w\cdot v$. Then by remark \ref{rem}, we have $T_{u}\left(x\right)=T_{w\cdot v}\left(x\right)=T_{w}\left(T_{v}\left(x\right)\right)=T_{w}\left(T_{v}\left(y\right)\right)=T_{w\cdot v}\left(y\right)=T_{u}\left(y\right)$.
\end{proof}
\begin{lem}
Let $S$ be a semigroup and $\left(X,\langle T_{s}\rangle_{s\in S}\right)$
be a dynamical system and let $x\in X$. Then there exists an idempotent
$u$ in $K\left(\overline{\mathcal{F}}\right)$ such that $T_{u}\left(x\right)=y$
if and only if both $y$ is $\mathcal{F}$-uniformly recurrent and
$x$ and $y$ are $\mathcal{F}$- proximal.
\end{lem}

\begin{proof}
Since $u$ is a minimal idempotent of $\overline{\mathcal{F}}$, there
is a minimal left ideal $L$ of $\overline{\mathcal{F}}$ such that
$u\in L$. By theorem \ref{4eq}, $y$ is $\mathcal{F}$-uniformly
recurrent and by remark \ref{rem}, $T_{u}(y)=T_{u}\left(T_{u}\left(x\right)\right)=T_{u\cdot u}\left(x\right)=T_{u}\left(x\right)$.
So $x$ and $y$ are $\mathcal{F}$- proximal.

Conversely, by lemma \ref{15}, pick a minimal left ideal $L$ of
$\overline{\mathcal{F}}$ such that $T_{u}\left(x\right)=T_{u}\left(y\right)$
for all $u\in L$. Pick by theorem \ref{4eq}, an idempotent $u\in L$
such that $T_{u}\left(y\right)=y$. Hence $T_{u}\left(x\right)=y.$
\end{proof}
\noindent We now give the dynamical characterization of $\mathcal{F}$-central
sets.
\begin{proof}[\textbf{\textit{Proof of Theorem \ref{main}:}}]
Let $G=S\cup\left\{ e\right\} $, $X=\prod_{s\in G}\left\{ 0,1\right\} $
and for $s\in S$ define $T_{s}:X\rightarrow X$ by $T_{s}\left(x\right)\left(t\right)=x\left(t\cdot s\right)$
for all $t\in G$. Then by \cite[Lemma 19.14]{key-4-2} $\left(X,\langle T_{s}\rangle_{s\in S}\right)$
is a dynamical system. Now let $x=\chi_{B}$, the characteristic function
of $B$. Pick a minimal idempotent in $\overline{\mathcal{F}}$ such
that $B\in u$ and let $y=T_{u}\left(x\right)$. Then by theorem \ref{4eq}
$y$ is $\mathcal{F}$-uniformly recurrent and $x$ and $y$ are $\mathcal{F}$-
proximal. Now let $U=\left\{ z\in X:z\left(e\right)=y\left(e\right)\right\} $.
Then $U$ is a neighbourhood of $y\in X$. We note that $y\left(e\right)=1$.
Indeed, $y=T_{u}\left(x\right)$ so, $\left\{ s\in S:T_{s}\left(x\right)\in U\right\} \in u$
and we may choose some $s\in B$ such that $T_{s}\left(x\right)\in U$.
Then $y\left(e\right)=T_{s}\left(x\right)\left(e\right)=x\left(s\cdot e\right)=1$.
Thus given any $s\in S$, $s\in B\Longleftrightarrow x\left(s\right)=1\Longleftrightarrow T_{s}\left(x\right)\in U$.

Conversely, choose a dynamical system $\left(X,\langle T_{s}\rangle_{s\in S}\right)$,
points $x,y\in X$ and a neighbourhood $U$ of $y$ such that $x$
and $y$ are $\mathcal{F}$- proximal, $y$ is $\mathcal{F}$-uniformly
recurrent and $B=\left\{ s\in S:T_{s}\left(x\right)\in U\right\} $.
Choose by theorem \ref{4eq}, a minimal idempotent $u\in\overline{\mathcal{F}}$
such that $T_{u}\left(x\right)=y$. Then $B\in u$.
\end{proof}

\noindent \textbf{Acknowledgment: }The author acknowledges the grant UGC-NET
SRF fellowship with id no. 421333 of CSIR-UGC NET December 2016.

\end{document}